\documentclass[12pt]{article}
\usepackage{amsfonts}
\usepackage{bezier}
\usepackage{euscript}
\usepackage{amsfonts,amssymb,amsmath,amsbsy,amsthm,amscd}%I included these styles
\usepackage[T2A]{fontenc}
\usepackage[english,russian]{babel}   % load Babel setup for English
\usepackage{graphicx}

\newtheorem{theorem}{╥хюЁхьр}
\newtheorem{lemma}{╦хььр}
\newtheorem{definition}{╬яЁхфхыхэшх}

\title{Immersions of graphs to the projective plane}
\author{M.\,A.~Ivashkovskii}
\date{}

\begin{document}

\maketitle

{\bf Abstract.}
Immersions of graphs to the projective plane are studied. A classification of immersions up to regular homotopy is obtained. A complete invariant of immersions up to regular homotopy is constructed. The case of graphs immersions to any compact surface different from the projective plane was known.

{\bf Key words:}
%{%╩╦▐╫┼┬█┼ ╤╦╬┬└ ═└ └═├╦╚╔╤╩╬╠ ▀╟█╩┼:
regular immersion of a graph, regular homotopy of immersions, Reidemeister moves, self-intersection number.
%}

\bigskip
\bigskip
\begin{center}
╧╬├╨╙╞┼═╚▀ ├╨└╘╬┬ ┬ ╧╨╬┼╩╥╚┬═╙▐ ╧╦╬╤╩╬╤╥▄ \bigskip \\
╠.\,└.~╚тр°ъютёъшщ
\end{center}
%\footnote[1]{%
%{\it ╚тр°ъютёъшщ ╠ръёшь └ыхъёрэфЁютшў} --- ёЄєф.\ 
%%фюъЄюЁ Їшч.-ьрЄ. эрєъ, яЁюЇ.
%ърЇ.\ фшЇЇхЁхэЎшры№эющ ухюьхЄЁшш ш яЁшыюцхэшщ ьхї.-ьрЄ.\ Ї-Єр ╠├╙,
%e-mail: frank1581@yandex.ru.}
%}

%\medskip
%{\bf └ээюЄрЎш :}
%{%└══╥╬═╓╚▀ ═└ ╨╙╤╤╩╬╠ ▀╟█╩┼:
\begin{abstract}
╚ёёыхфє■Єё  яюуЁєцхэш  уЁрЇют т яЁюхъЄштэє■ яыюёъюёЄ№. ╧юыєўхэр ъырёёшЇшърЎш  яюуЁєцхэшщ ё ЄюўэюёЄ№■ фю Ёхуєы Ёэющ уюьюЄюяэюёЄш. ╧юёЄЁюхэ яюыэ√щ шэтрЁшрэЄ яюуЁєцхэшщ ё ЄюўэюёЄ№■ фю Ёхуєы Ёэющ уюьюЄюяэюёЄш. ╤ыєўрщ яюуЁєцхэшщ уЁрЇют т ы■сє■ ъюьяръЄэє■ яютхЁїэюёЄ№, юЄышўэє■ юЄ яЁюхъЄштэющ яыюёъюёЄш, с√ы шчтхёЄхэ.
%}

{\bf ╩ы■ўхт√х ёыютр:}
%{%╩╦▐╫┼┬█┼ ╤╦╬┬└ ═└ ╨╙╤╤╩╬╠ ▀╟█╩┼:
Ёхуєы Ёэюх яюуЁєцхэшх уЁрЇр, Ёхуєы Ёэр  уюьюЄюяш  яюуЁєцхэшщ, фтшцхэш  ╨хщфхьхщёЄхЁр, ўшёыю ёрьюяхЁхёхўхэш .
%}
\end{abstract}

\bigskip
%------------------------------------------------------------------

╧єёЄ№ фрэ ёт чэ√щ уЁрЇ $G$ (тючьюцэю, шьх■∙шщ яхЄыш ш ъЁрЄэ√х ЁхсЁр) ё т√фхыхээющ эр эхь тхЁ°шэющ $v$. 
╨рёёьюЄЁшь яюуЁєцхэшх $\gamma: G \looparrowright M$ уЁрЇр $G$ т ёт чэюх ъюьяръЄэюх уырфъюх фтєьхЁэюх ьэюуююсЁрчшх $M$ (юяЁхфхыхэшх 1 эшцх). ╥ЁхсєхЄё  яюыєўшЄ№ ъырёёшЇшърЎш■ тёхї тючьюцэ√ї яюуЁєцхэшщ ё ЄюўэюёЄ№■ фю Ёхуєы Ёэющ уюьюЄюяэюёЄш (юяЁхфхыхэшх 2 эшцх). ┬ эрёЄю ∙хщ чрьхЄъх шчырурхЄё  Ёх°хэшх [1] ¤Єющ чрфрўш т ёыєўрх, ъюуфр $M=\mathbb R P^2$ --- яЁюхъЄштэр  яыюёъюёЄ№, ш яюёЄЁюхэ яюыэ√щ шэтрЁшрэЄ яюуЁєцхэшщ уЁрЇр $G$ т яЁюхъЄштэє■ яыюёъюёЄ№ т ЄхЁьшэрї шэфхъёр ёрьюяхЁхёхўхэш  ъЁшт√ї яю ьюфєы■ 2.
╧юыэ√щ шэтрЁшрэЄ яюуЁєцхэшщ уЁрЇр $G$ т ы■сє■ яютхЁїэюёЄ№ $M\ne\mathbb R P^2$ с√ы яюёЄЁюхэ ─.\,└.~╧хЁь ъют√ь [2] т ЄхЁьшэрї ўшёхы тЁр∙хэш  ъЁшт√ї.

╧хЁхщфхь ъ Єюўэ√ь ЇюЁьєышЁютърь. 
┴єфхь яЁхфяюырурЄ№, ўЄю уЁрЇ $G$ ёюёЄюшЄ шч юфэющ тхЁ°шэ√ ш $n$ ЁхсхЁ (яхЄхы№) $e_i$, $i=1,\dots,n$, Є.х.\  ты хЄё  сєъхЄюь $n$ юъЁєцэюёЄхщ. ╤ыєўрщ яЁюшчтюы№эюую уЁрЇр ыхуъю ётюфшЄё  ъ ¤Єюьє ёыєўр■ [2].
┬√схЁхь ш ЇшъёшЁєхь ярЁрьхЄЁшчрЎш■ эр ърцфюь ЁхсЁх $e_i$ уЁрЇр $G$.

\begin{definition}%{╬яЁхфхыхэшх 1.} 
╬ЄюсЁрцхэшх $\gamma:G \to M$ эрчютхь {\em яюуЁєцхэшхь} уЁрЇр $G$ т яютхЁїэюёЄ№ $M$ (ш юсючэрўшь ўхЁхч $\gamma:G \looparrowright M$), хёыш хую юуЁрэшўхэшх эр ы■сюх чрьъэєЄюх ЁхсЁю $\overline{e_i}$  ты хЄё  Ёхуєы Ёэющ ъЁштющ (юЄэюёшЄхы№эю ярЁрьхЄЁшчрЎшш $t\mapsto(\gamma|_{e_i})(t)$, $t\in[0,1]$, юЄтхўр■∙хщ чрфрээющ ярЁрьхЄЁшчрЎшш ЁхсЁр $e_i$), ш ъЁюьх Єюую $2n$ ърёрЄхы№э√ї тхъЄюЁют $\frac{d}{dt}|_{t=0}(\gamma|_{\overline{e_i}}),-\frac{d}{dt}|_{t=1}(\gamma|_{\overline{e_i}})$, $i=1,\dots,n$, т Єюўъх $\gamma(v)$ яюярЁэю эхёюэряЁртыхэ√, уфх $v$ --- тхЁ°шэр, т ъюЄюЁющ ёїюф Єё  яхЄыш уЁрЇр $G$.
\end{definition}

\begin{definition}%{╬яЁхфхыхэшх 2.}
╤хьхщёЄтю яюуЁєцхэшщ $\gamma_u: G\looparrowright M$, $u\in[0,1]$, уЁрЇр т яютхЁїэюёЄ№ эрчютхь {\em Ёхуєы Ёэющ уюьюЄюяшхщ}, хёыш юэю  ты хЄё  уюьюЄюяшхщ т юс√ўэюь ёь√ёых (Є.х.\ 
юЄюсЁрцхэшх $\Gamma:G\times[0,1]\to M$, $\Gamma(x,u)=\gamma_u(x)$, эхяЁхЁ√тэю), ш юуЁрэшўхэшх ¤Єющ уюьюЄюяшш эр ы■сюх чрьъэєЄюх ЁхсЁю $\overline{e_i}$ чрфрхЄё  $C^\infty$-уырфъшь юЄюсЁрцхэшхь $[0,1]\times[0,1]\to M$, $(t,u)\mapsto(\gamma_u|_{\overline{e_i}})(t)$, $t,u\in[0,1]$.
─тр яюуЁєцхэш  $\gamma_0,\gamma_1$ эрчютхь {\em Ёхуєы Ёэю уюьюЄюяэ√ьш} (юсючэрўхэшх $\gamma_0 \overset{\text{\it reg}}{\sim} \gamma_1$), хёыш шї ьюцэю ёюхфшэшЄ№ Ёхуєы Ёэющ уюьюЄюяшхщ $\gamma_u: G\looparrowright M$, $u\in[0,1]$.
\end{definition}

╬яш°хь яюёЄЁюхэшх яюыэюую шэтрЁшрэЄр Ёхуєы Ёэющ уюьюЄюяэюёЄш яюуЁєцхэшщ.

\begin{definition}%{╬яЁхфхыхэшх 3.}
╬сючэрўшь $\gamma|_{e_i}=:e'_1$, $i=1,\dots,n$. ╚ч тхЁ°шэ√ $v':=\gamma(v)$ уЁрЇр $G':=\gamma(G)$ т√їюф Є яюыєЁхсЁр $e'_1,\dots,e'_n$ ш т эхх тїюф Є яюыєЁхсЁр $(e'_1)^{-1},\dots,(e'_n)^{-1}$. 
╧єёЄ№ $or$ --- ыюъры№эр  юЁшхэЄрЎш  ърёрЄхы№эющ яыюёъюёЄш $T_{\gamma(v)}M$ т Єюўъх $\gamma(v)$.
┬√схЁхь юсїюф тхЁ°шэ√ $v'$, ёюуырёютрээ√щ ё юЁшхэЄрЎшхщ $or$. ┴єфхь т√яшё√трЄ№ яюыєЁхсЁр $e'_i,(e'_i)^{-1}$ т яюЁ фъх шї яю тыхэш  тю тЁхь  юсїюфр. ╨рёёьюЄЁшь ёююЄтхЄёЄтє■∙шщ эрсюЁ шч $e_i,e_i^{-1}$ (Є.х.\ ёююЄтхЄёЄтє■∙хх ёыютю т рыЇртшЄх $\{e_1,\dots,e_n,e_1^{-1},\dots,e_n^{-1}\}$) ё ЄюўэюёЄ№■ фю Ўшъышўхёъющ яхЁхёЄрэютъш ¤ыхьхэЄют эрсюЁр эрчютхь {\em Ўшъышўхёъшь яюЁ фъюь яюыєЁхсхЁ т тхЁ°шэх}, юЄтхўр■∙шь яюуЁєцхэш■ $\gamma$ ш ыюъры№эющ юЁшхэЄрЎшш $or$.
╧Ёш шчьхэхэшш ыюъры№эющ юЁшхэЄрЎшш яюыєўшь фЁєующ Ўшъышўхёъшщ яюЁ фюъ; яюыєўхээє■ эхєяюЁ фюўхээє■ ярЁє Ўшъышўхёъшї яюЁ фъют эрчютхь {\em Ўшъышўхёъшь яюЁ фъюь яюыєЁхсхЁ т тхЁ°шэх} фы  яюуЁєцхэш  $\gamma$.
\end{definition}

╚ьххь сшхъЎш■ ьхцфє ьэюцхёЄтюь тёхтючьюцэ√ї Ўшъышўхёъшї яюЁ фъют фы  юёэр∙хээ√ї яюуЁєцхэшщ ш ьэюцхёЄтюь яхЁхёЄрэютюъ $\Sigma_{2n-1}$ яюЁ фър $2n-1$. ╚чьхэхэшх ыюъры№эющ юЁшхэЄрЎшш фрхЄ фхщёЄтшх (ЄЁштшры№эюх яЁш $n=1$, ётюсюфэюх яЁш $n>1$) уЁєяя√ $\mathbb Z_2$ эр ьэюцхёЄтх $\Sigma_{2n-1}$. ╧юыєўрхь сшхъЎш■ ьхцфє ьэюцхёЄтюь тёхтючьюцэ√ї Ўшъышўхёъшї яюЁ фъют фы  яюуЁєцхэшщ ш яЁюёЄЁрэёЄтюь юЁсшЄ $\Sigma_{2n-1}/\mathbb Z_2$.

╥ръ ъръ ЇєэфрьхэЄры№эр  уЁєяяр яЁюхъЄштэющ яыюёъюёЄш рсхыхтр, юэр эх чртшёшЄ юЄ т√сюЁр срчшёэющ Єюўъш. ╧ю¤Єюьє сєфхь юяєёърЄ№ срчшёэє■ Єюўъє т юсючэрўхэшш ЇєэфрьхэЄры№эющ уЁєяя√. 

\begin{definition}%{╬яЁхфхыхэшх 4 {\rm(
[{\it шэфхъё ёрьюяхЁхёхўхэш } чрьъэєЄющ Ёхуєы Ёэющ ъЁштющ эр яютхЁїэюёЄш]%)}.}
╧єёЄ№ $\gamma:[0,1]\looparrowright M$ --- чрьъэєЄр  Ёхуєы Ёэр  ъЁштр  (с√Є№ ьюцхЄ, шьх■∙р  ``шчыюь'' т Єюўъх $\gamma(0)=\gamma(1)$), ш т Єюўъх $\gamma(0)$ ЇшъёшЁютрэр юЁшхэЄрЎш  ърёрЄхы№эющ яыюёъюёЄш $T_{\gamma(0)}M$, яЁшўхь тхъЄюЁ√ $\dot\gamma(0)$ ш $\dot\gamma(1)$ эх  ты ■Єё  яЁюЄштюяюыюцэю эряЁртыхээ√ьш. ╧Ёхфяюыюцшь, ўЄю тёх Єюўъш ёрьюяхЁхёхўхэш  ¤Єющ ъЁштющ  ты ■Єё  ЄЁрэётхЁёры№э√ьш (¤Єюую ьюцэю фюсшЄ№ё  ьрыющ фхЇюЁьрЎшхщ ъЁштющ). ╧єёЄ№ $(t_1,t_2)$ --- Єюўър ёрьюяхЁхёхўхэш  ъЁштющ $\gamma$, Є.х.\ $\gamma(t_1)=\gamma(t_2)=:A$, яЁшўхь $0\leq t_1<t_2\leq 1$ ш $(t_1,t_2)\ne(0,1)$. ╧хЁхэхёхь юЁшхэЄрЎш■ тфюы№ ъЁштющ $\gamma|_{[0,t_1]}$ т Єюўъє $A$. ┬юч№ьхь ЁхяхЁ $(\dot\gamma(t_1),\dot\gamma(t_2))$ ш ёЁртэшь хую юЁшхэЄрЎш■ ё яхЁхэхёхээющ юЁшхэЄрЎшхщ. ┬ чртшёшьюёЄш юЄ ёюуырёютрээюёЄш яЁшяш°хь Єюўъх ёрьюяхЁхёхўхэш  $(t_1,t_2)$ чэръ $+1$ шыш $-1$. ╧ЁюёєььшЁютрт ¤Єш ўшёыр яю тёхь Єюўърь ёрьюяхЁхёхўхэш , яюыєўшь {\it шэфхъё ёрьюяхЁхёхўхэш } ъЁштющ $\gamma$ юЄэюёшЄхы№эю чрфрээющ ыюъры№эющ юЁшхэЄрЎшш, юсючэрўшь хую ўхЁхч $I(\gamma)\in\mathbb Z$. ╟рьхЄшь, ўЄю ўшёыю $I(\gamma)\ mod \ 2\in\mathbb Z_2$ эх чртшёшЄ юЄ т√сюЁр ыюъры№эющ юЁшхэЄрЎшш.
\end{definition}

\begin{theorem}%{╥хюЁхьр.}
╧єёЄ№ $\gamma_1,\gamma_2$ --- фтр яюуЁєцхэш  уЁрЇр $G$ т яЁюхъЄштэє■ яыюёъюёЄ№ $\mathbb RP^2$. ▌Єш яюуЁєцхэш  Ёхуєы Ёэю уюьюЄюяэ√ (Є.х.\ $\gamma_1 \overset{\text{\it reg}}{\sim} \gamma_2$) Єюуфр ш Єюы№ъю Єюуфр, ъюуфр $Inv(\gamma_1) = Inv(\gamma_2)$.
╟фхё№ ЇєэъЎшюэры $Inv:\{\gamma:G\looparrowright \mathbb RP^2\}\to(\Sigma_{2n-1}/\mathbb Z_2)\times \{0,1\}^{2n}$ юяЁхфхы хЄё  ЇюЁьєырьш
$$
Inv_1(\gamma):=\Big( \mbox{Ўшъышўхёъшщ яюЁ фюъ яюыєЁхсхЁ фы  $\gamma$} \Big)\in\Sigma_{2n-1}/\mathbb Z_2,
$$ 
$$
Inv_2(\gamma):=\Big([\gamma|_{e_1}],\dots,[\gamma|_{e_n}]\Big) \in \Big(\pi_1 (\mathbb RP^2)\Big)^n \cong (\mathbb Z_2)^n=\{0,1\}^n,
$$
$$
Inv_3(\gamma):=\Big(I (\gamma|_{e_1}) \ mod \ 2,\dots,I (\gamma|_{e_n}) \ mod \ 2
\Big) \in (\mathbb Z_2)^n=\{0,1\}^n,
$$
$$
Inv(\gamma):=\Big(Inv_1(\gamma),Inv_2(\gamma),Inv_3(\gamma)\Big).
$$
\end{theorem}

╥хюЁхьр ёыхфєхЄ шч ыхьь 1 ш 3, яЁштхфхээ√ї эшцх.
─ы  ърцфющ яхЄыш $e_i$, $i=1,\dots,n$, уЁрЇр $G$ тючьюцэ√ фтр ёыєўр .
╧хЁт√щ --- ъюуфр яхЄы  $\gamma_1|_{e_i}$ ёЄ уштрхьр. ┬ЄюЁющ --- ъюуфр ¤Єр яхЄы  эхёЄ уштрхьр.

\begin{figure}[htbp]
\includegraphics[width=0.20\textwidth]{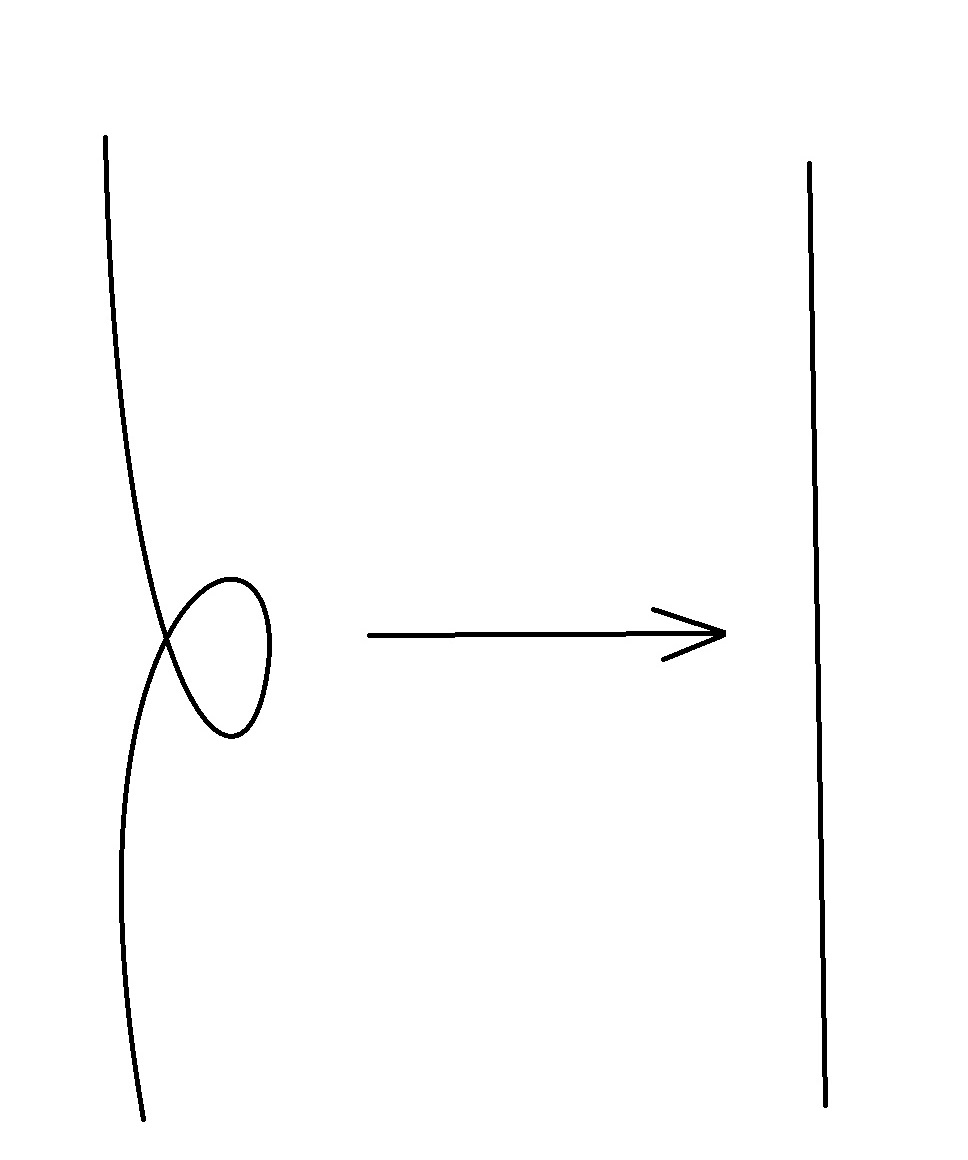}
\includegraphics[width=0.30\textwidth]{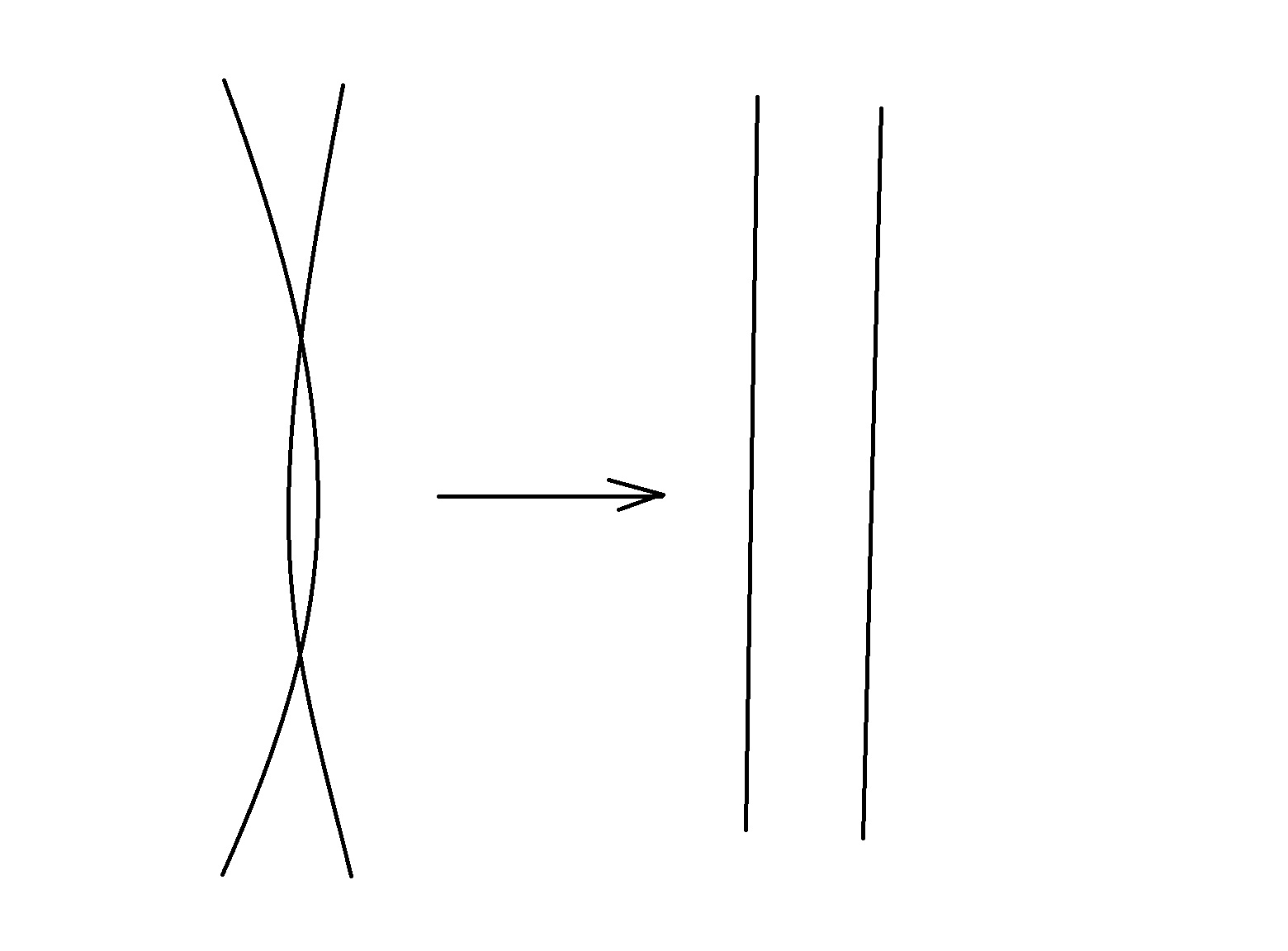} 
\includegraphics[width=0.36\textwidth]{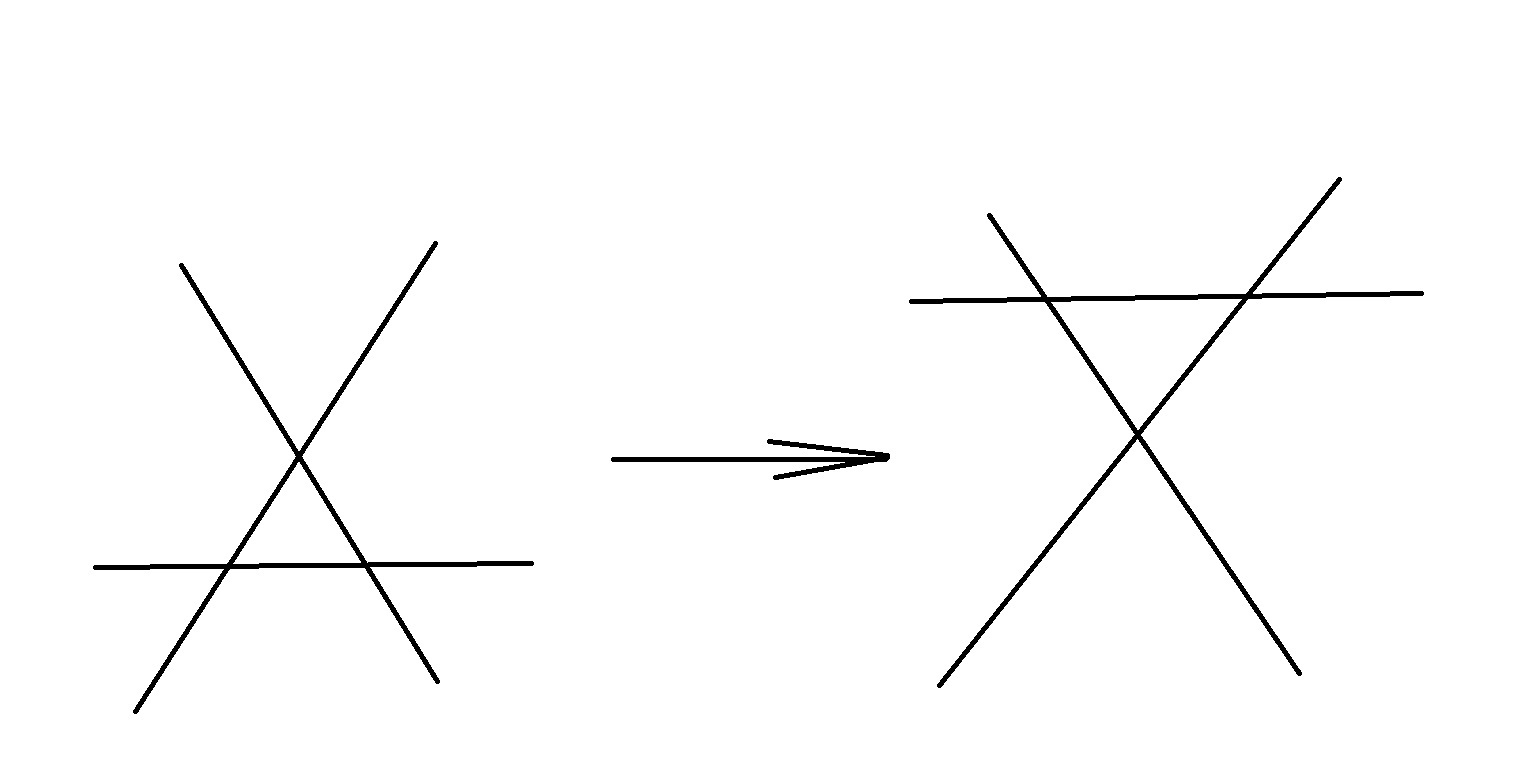}
\caption {─тшцхэш  ╨хщфхьхщёЄхЁр Єшяют 1, 2 ш 3} \label {fig:Reid}
\end {figure}

╟рьъэєЄє■ ъЁштє■ $\gamma:S^1\to M$ эрчютхь {\em ъЁштющ юс∙хую яюыюцхэш }, хёыш юэр Ёхуєы Ёэр (ш, тючьюцэю, шьххЄ ``шчыюь'' т тхЁ°шэх), тёх хх Єюўъш ёрьюяхЁхёхўхэш  (ёь.\ юяЁхфхыхэшх 4)  ты ■Єё  ЄЁрэётхЁёры№э√ьш, ш юсЁрч√ $\gamma(t_1)=\gamma(t_2)$ хх Єюўхъ ёрьюяхЁхёхўхэш  $(t_1,t_2)$ яюярЁэю Ёрчышўэ√.
╤юуырёэю шчтхёЄэющ ЄхюЁхьх ╨хщфхьхщёЄхЁр [3], фтх чрьъэєЄ√х ъЁшт√х юс∙хую яюыюцхэш  $\gamma_0,\gamma_1:[0,1]\to M$ эр яютхЁїэюёЄш $M$ уюьюЄюяэ√ Єюуфр ш Єюы№ъю Єюуфр, ъюуфр ёє∙хёЄтєхЄ уюьюЄюяш , Ёрчырур■∙р ё  т яюёыхфютрЄхы№эюёЄ№ фтшцхэшщ ╨хщфхьхщёЄхЁр 1-ую, 2-ую ш 3-ую Єшяют 
% \marginpar {Ёшё.~1 Ёшё.~1 Ёшё.~1 Ёшё.~1 Ёшё.~\ref {fig:Reid}}
(Ёшё.~\ref {fig:Reid}).

\begin{lemma}%{╦хььр 1.}
┼ёыш $\gamma_1 \overset{\text{\it reg}}{\sim} \gamma_2$, Єю $Inv(\gamma_1)=Inv(\gamma_2)$.
\end{lemma}

%{\bf ─юърчрЄхы№ёЄтю.}
\begin{proof}
┼ёыш яюуЁєцхэш  $\gamma_1$ ш $\gamma_2$ Ёхуєы Ёэю уюьюЄюяэ√, Єю юэш ыюъры№эю Ёхуєы Ёэю уюьюЄюяэ√ т эхъюЄюЁющ юъЁхёЄэюёЄш тхЁ°шэ√ $v$. ╟эрўшЄ, Ўшъышўхёъшщ яюЁ фюъ юЄэюёшЄхы№эю юЁшхэЄрЎшш ьхэ Є№ё  эх сєфхЄ.

┼ёыш яюуЁєцхэш  $\gamma_1$ ш $\gamma_2$ Ёхуєы Ёэю уюьюЄюяэ√, Єю фтх яхЄыш $\gamma_1|_{e_i}$ ш $\gamma_2|_{e_i}$ Ёхуєы Ёэю уюьюЄюяэ√, р яюЄюьє уюьюЄюяэ√. ╟эрўшЄ, юэш ышсю юсх ёЄ уштрхь√, ышсю юсх эхёЄ уштрхь√.

─ы  чрьъэєЄющ Ёхуєы Ёэющ ъЁштющ шэфхъё ёрьюяхЁхёхўхэш  $mod \ 2$ яЁш Ёхуєы Ёэющ уюьюЄюяшш ьхэ Є№ё  эх сєфхЄ.
▌Єю ёыхфєхЄ шч ``Ёхуєы Ёэюую рэрыюур'' ЄхюЁхь√ ╨хщфхьхщёЄхЁр [3]: шэфхъё ёрьюяхЁхёхўхэш  ъЁштющ ёюїЁрэ хЄё  яЁш фтшцхэш ї ╨хщфхьхщёЄхЁр тЄюЁюую ш ЄЁхЄ№хую Єшяют, р яюЄюьє ш яЁш ы■сющ Ёхуєы Ёэющ уюьюЄюяшш.

╥ръшь юсЁрчюь, хёыш $\gamma_1$ ш $\gamma_2$ Ёхуєы Ёэю уюьюЄюяэ√, Єю чэрўхэш  ттхфхээюую ЇєэъЎшюэрыр эр эшї ёютярфр■Є. ╥ю хёЄ№, ЇєэъЎшюэры шэтрЁшрэЄхэ юЄэюёшЄхы№эю Ёхуєы Ёэющ уюьюЄюяэюёЄш яюуЁєцхэшщ.
%\qed
\end{proof}

╥хяхЁ№ яЁютхфхь фюърчрЄхы№ёЄтю т юсЁрЄэє■ ёЄюЁюэє.

\begin{lemma}{╦хььр 2.}
┼ёыш чрьъэєЄ√х Ёхуєы Ёэ√х ъЁшт√х $\gamma_0,\gamma_1:S^1\looparrowright M$ эр яютхЁїэюёЄш $M$ уюьюЄюяэ√, Єю $\gamma_0 \overset{\text{\it reg}}{\sim} \tilde\gamma_1$ 
фы  эхъюЄюЁющ Ёхуєы Ёэющ ъЁштющ $\tilde\gamma_1$, яюыєўр■∙хщё  шч ъЁштющ $\gamma_1$ яЁшсртыхэшхь эхъюЄюЁюую ъюышўхёЄтр ьрыхэ№ъшї яхЄхы№ (чртшЄъют). ┼ёыш яЁш ¤Єюь ъЁшт√х $\gamma_0,\gamma_1$ ёютярфр■Є эр эхъюЄюЁющ фєух $\delta\subset S^1$, Єю єърчрээр  Ёхуєы Ёэр  уюьюЄюяш  ьюцхЄ с√Є№ т√сЁрэр эхяюфтшцэющ эр фєух $\delta$.
\end{lemma}

%{\bf ─юърчрЄхы№ёЄтю.} 
\begin{proof}
╧ю ЄхюЁхьх ╨хщфхьхщёЄхЁр [3], ттшфє уюьюЄюяэюёЄш ъЁшт√ї, ёє∙хёЄтєхЄ яюёыхфютрЄхы№эюёЄ№ фтшцхэшщ ╨хщфхьхщёЄхЁр, яЁштюф ∙р  ъЁштє■ $\gamma_0$ ъ ъЁштющ $\gamma_1$. ┬√яюыэшь ¤Єє яюёыхфютрЄхы№эюёЄ№ фтшцхэшщ ёю ёыхфє■∙хщ ьюфшЇшърЎшхщ: ърцф√щ Ёрч, ъюуфр яЁхфяюырурхЄё  ёютхЁ°шЄ№ фтшцхэшх ╨хщфхьхщёЄхЁр Єшяр 1 (Ёюцфхэшх шыш єэшўЄюцхэшх чртшЄър), ь√ тьхёЄю хую т√яюыэхэш  сєфхь юёЄрты Є№ эр ьрыюь єўрёЄъх ъЁштющ ьрыхэ№ъшщ чртшЄюъ (ъюЄюЁ√щ сєфхЄ ёюїЁрэ Є№ё , цхёЄъю фтшур ё№ яю яютхЁїэюёЄш тьхёЄх ё ¤Єшь єўрёЄъюь ъЁштющ, т Єхўхэшх тёхщ яюёыхфє■∙хщ Ёхуєы Ёэющ уюьюЄюяшш).

\begin{figure}[htbp]
\begin{center}
\includegraphics[width=0.6\textwidth]{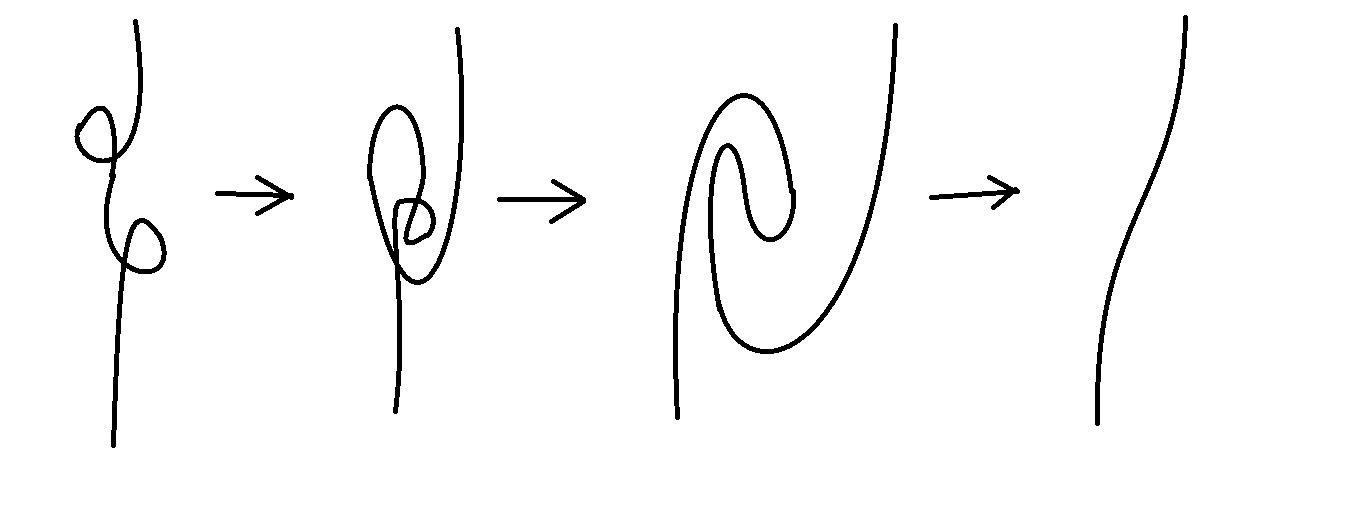}
\end{center}
\caption {┬чршьэюх єэшўЄюцхэшх ярЁ√ ьрыхэ№ъшї яхЄхы№} \label {fig:pair}
\end {figure}

┴юыхх яюфЁюсэю: хёыш фтшцхэшх ╨хщфхьхщёЄхЁр Єшяр 1 ёюёЄюшЄ т {\it єэшўЄюцхэшш} чртшЄър, Єю тьхёЄю хую єэшўЄюцхэш  ь√ ёюїЁрэшь хую (т Єхўхэшх тёхщ яюёыхфє■∙хщ Ёхуєы Ёэющ уюьюЄюяшш) т тшфх ьрыхэ№ъюую чртшЄър. └ хёыш фтшцхэшх ╨хщфхьхщёЄхЁр Єшяр 1 ёюёЄюшЄ т {\it Ёюцфхэшш} чртшЄър, Єю ь√ эхяюёЁхфёЄтхээю яхЁхф т√яюыэхэшхь ¤Єюую фтшцхэш  яюЁюфшь (ъръ эр 
% \marginpar {Ёшё.~\ref {fig:pair}}
Ёшё.~\ref {fig:pair}) Ёхуєы Ёэющ уюьюЄюяшхщ ярЁє сышчыхцр∙шї чртшЄъют яЁюЄштюяюыюцэ√ї чэръют, ш юфшэ шч эшї (эєцэюую чэрър) шёяюы№чєхь т ърўхёЄтх Ёюцфхээюую чртшЄър яЁш фтшцхэшш ╨хщфхьхщёЄхЁр, р фЁєующ чртшЄюъ ёюїЁрэшь (т Єхўхэшх тёхщ яюёыхфє■∙хщ Ёхуєы Ёэющ уюьюЄюяшш) т тшфх ьрыхэ№ъюую чртшЄър эр фрээюь єўрёЄъх ъЁштющ.

┬ Ёхчєы№ЄрЄх ь√ яюыєўшь Ёхуєы Ёэє■ уюьюЄюяш■, яЁхюсЁрчє■∙є■ ъЁштє■ $\gamma_0$ ъ эхъюЄюЁющ ъЁштющ $\tilde\gamma_1$ шёъюьюую тшфр, Є.х.\ яюыєўр■∙хщё  шч $\gamma_1$ фюсртыхэшхь эхёъюы№ъшї чртшЄъют.
%\qed
\end{proof}

\begin{figure}[htbp]
\begin{center}
\includegraphics[width=0.45\textwidth]{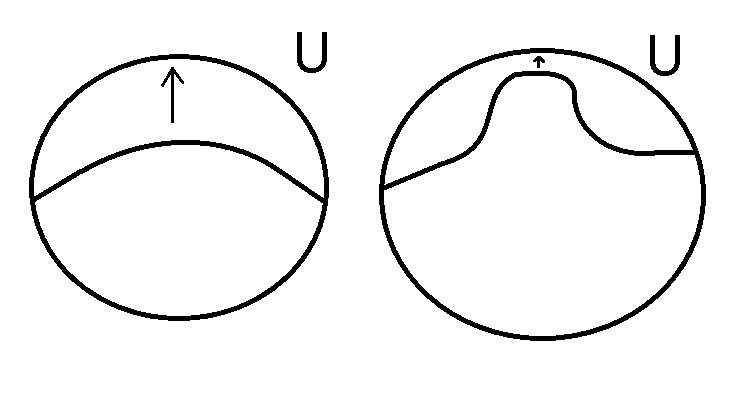}
\hspace*{2mm}
\includegraphics[width=0.45\textwidth]{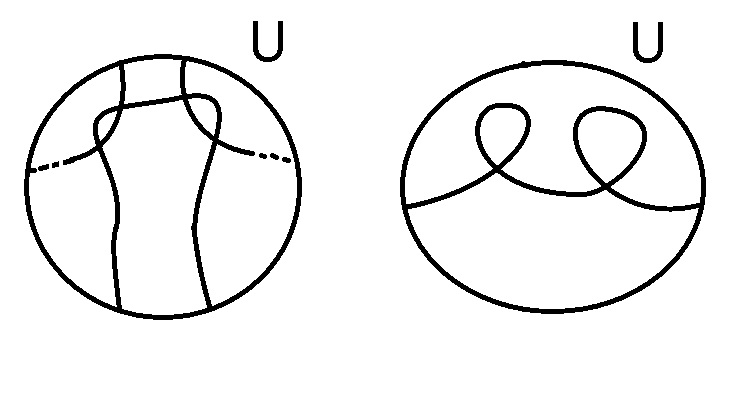}
\end{center}
\caption {``╧ЁюЄрёъштрэшх ўхЁхч схёъюэхўэюёЄ№''} \label {fig:protask}
\end {figure}

─ы  чрьъэєЄющ Ёхуєы Ёэющ ъЁштющ $\gamma$ эр $\mathbb RP^2$ ттхфхь фтшцхэшх, ъюЄюЁюх сєфхь эрч√трЄ№ \textit{яЁюЄрёъштрэшхь ўхЁхч схёъюэхўэюёЄ№}. ─ы  ¤Єюую яхЁхщфхь ъ эръЁ√Єш■ яЁюхъЄштэющ яыюёъюёЄш $\mathbb RP^2$ ёЇхЁющ $S^2$ ш яЁюфхЇюЁьшЁєхь эхъюЄюЁє■ яЁюёЄє■ фєує $\gamma|_{\delta}$, $\delta \not\ni v$, яЁш яюью∙ш Ёхуєы Ёэющ уюьюЄюяшш, эхяюфтшцэющ тэх $\delta$ ш яЁхюсЁрчє■∙хщ фєує $\gamma|_\delta$ т фєує тшфр $\gamma|_\delta$ ё фтєь  яхЄы ьш юфэюую чэрър, ъръ яюърчрэю эр 
% \marginpar {Ёшё.~\ref {fig:protask}}
Ёшё.~\ref {fig:protask}.
╚Єръ, ё яюью∙№■ фрээюую фтшцхэш  ьюцэю ьхэ Є№ шэфхъё ёрьюяхЁхёхўхэш  ы■сющ чрьъэєЄющ Ёхуєы Ёэющ ъЁштющ $\gamma$ т $\mathbb RP^2$ эр $+2$ шыш $-2$ яЁш яюью∙ш Ёхуєы Ёэющ уюьюЄюяшш, эхяюфтшцэющ т эхъюЄюЁющ юъЁхёЄэюёЄш ъюэЎют ъЁштющ.

\begin{lemma}%{╦хььр 3.}
┼ёыш $Inv(\gamma_1)=Inv(\gamma_2)$, Єю $\gamma_1 \overset{\text{\it reg}}{\sim} \gamma_2$.
\end{lemma}

%{\bf ─юърчрЄхы№ёЄтю.} 
\begin{proof}
╤ютьхёЄшь Єюўъш $\gamma_1(v)=v'_1$ ш $\gamma_2(v)=v'_2$ Єръ, ўЄюс√ ёютярыш Ўшъышўхёъшх яюЁ фъш т тхЁ°шэх $v$ фы  яюуЁєцхэшщ $\gamma_1$ ш $\gamma_2$. ╥ръ ъръ Ўшъышўхёъшщ яюЁ фюъ юфшэ ш ЄюЄ цх, Єю ёрьш ¤Єш яюуЁєцхэш  ьюцэю ёютьхёЄшЄ№ т ьрыющ юъЁхёЄэюёЄш $U \subseteq G$ тхЁ°шэ√ $v$ яєЄхь Ёхуєы Ёэющ уюьюЄюяшш, эхяюфтшцэющ т тхЁ°шэх $v$.

╧єёЄ№ $U'\subset M=\mathbb R P^2$ --- ьрыр  ъЁєуютр  юъЁхёЄэюёЄ№ тхЁ°шэ√ $v'=v_1'=v_2'$ уЁрЇр $G_1'=\gamma_1(G)$. ╘шъёшЁєхь ыюъры№эє■ юЁшхэЄрЎш■ ърёрЄхы№эющ яыюёъюёЄш $T_{v'}\mathbb R P^2$.
╘шъёшЁєхь ы■сюх ЁхсЁю $e_i$ уЁрЇр $G$, $i=1,\dots,n$. ╬сючэрўшь ўхЁхч $I_i$ шэфхъё ёрьюяхЁхёхўхэш  $I(\gamma_1|_{e_i}) \ mod \ 2=I(\gamma_2|_{e_i}) \ mod \ 2\in\{0,1\}$, ъюЄюЁ√щ сєфхь ЁрёёьрЄЁштрЄ№ ъръ Ўхыюх ўшёыю. ┬ючьюцэ√ фтр ёыєўр .

{\em ╤ыєўрщ 1:} яхЄы  $\gamma_1|_{e_i}$ ёЄ уштрхьр.
╚ч-чр ёЄ уштрхьюёЄш яхЄыш $\gamma_1|_{e_i}$ ёє∙хёЄтєхЄ яюёыхфютрЄхы№эюёЄ№ фтшцхэшщ ╨хщфхьхщёЄхЁр Єшяют 2 ш 3, фхЇюЁьшЁє■∙шї ¤Єє яхЄы■ ъ эхъюЄюЁющ яхЄых $\tilde{\gamma}_1|_{e_i} \subset U' \subset \mathbb RP^2$, яЁшўхь ¤Єр фхЇюЁьрЎш  эхяюфтшцэр т эхъюЄюЁющ юъЁхёЄэюёЄш $V\subset U$ тхЁ°шэ√ $v$ уЁрЇр $G$.
└эрыюушўэю ёє∙хёЄтєхЄ яюёыхфютрЄхы№эюёЄ№ фтшцхэшщ ш фы  $\gamma_2$ Єрър , ўЄю $\tilde{\gamma}_2(e_i) \subset U'$.
╧хЄы  $\tilde\gamma_1|_{e_i}$ яю ыхььх 2 Ёхуєы Ёэю уюьюЄюяэр т фшёъх $U'$ яюыюцшЄхы№эю юЁшхэЄшЁютрээющ юъЁєцэюёЄш ё эхъюЄюЁ√ь ўшёыюь чртшЄъют, яЁшўхь ёююЄтхЄёЄтє■∙р  Ёхуєы Ёэр  уюьюЄюяш  эхяюфтшцэр эр $V\cap e_i$. 

╦■сє■ ярЁє ёюёхфэшї чртшЄъют, шьх■∙шї юфшэ чэръ, тчршьэю єэшўЄюцшь Ёхуєы Ёэющ уюьюЄюяшхщ, юсЁрЄэющ ъ ``яЁюЄрёъштрэш■ ўхЁхч схёъюэхўэюёЄ№'' (Ёшё.~\ref {fig:protask}).
└ ы■сє■ ярЁє ёюёхфэшї чртшЄъют ё яЁюЄштюяюыюцэ√ьш чэрърьш ьюцэю тчршьэю єэшўЄюцшЄ№ Ёхуєы Ёэющ уюьюЄюяшхщ (эхяюфтшцэющ эр $V\cap e_i$) т фшёъх $U'$, шчюсЁрцхээющ эр Ёшё.~\ref {fig:pair}. ╬ЄьхЄшь, ўЄю тёх єърчрээ√х уюьюЄюяшш эхяюфтшцэ√ эр $V\cap e_i$.
╥ръшь юсЁрчюь, яхЄы■ $\tilde\gamma_1|_{e_i}$ ьюцэю яЁхюсЁрчютрЄ№ ё яюью∙№■ Ёхуєы Ёэющ уюьюЄюяшш, эхяюфтшцэющ эр $V\cap e_i$, т яюыюцшЄхы№эю юЁшхэЄшЁютрээє■ юъЁєцэюёЄ№ ё 0 шыш 1 яюыюцшЄхы№э√ь чртшЄъюь. ╟рьхЄшь, ўЄю ъюышўхёЄтю чртшЄъют Ёртэю $I_i$ (Єръ ъръ шэфхъё ёрьюяхЁхёхўхэш  яю ьюфєы■ 2 Ёхуєы Ёэющ чрьъэєЄющ ъЁштющ ёюїЁрэ хЄё  яЁш тЄюЁюь ш ЄЁхЄ№хь фтшцхэш ї ╨хщфхьхщёЄхЁр, р яюЄюьє ш яЁш ы■сющ Ёхуєы Ёэющ уюьюЄюяшш). 
╧ю¤Єюьє ъ ¤Єюьє цх тшфє ьюцэю яЁхюсЁрчютрЄ№ ш яхЄы■ $\tilde\gamma_2|_{e_i}$.

╬Єё■фр ёыхфєхЄ, ўЄю яхЄыш $\gamma_1|_{e_i}$ ш $\gamma_2|_{e_i}$ Ёхуєы Ёэю уюьюЄюяэ√ ё яюью∙№■ Ёхуєы Ёэющ уюьюЄюяшш, эхяюфтшцэющ эр $V\cap e_i$.

{\em ╤ыєўрщ 2:} яхЄы  $\gamma_1|_{e_i}$ эхёЄ уштрхьр. ┬ фрээюь ёыєўрх яхЄы  $\gamma_1|_{e_i}$ яхЁхёхърхЄ ыхэЄє ╠хсшєёр $\mathbb R P^2\setminus U'$.
╧хЄы  $\gamma_1|_{e_i}$ яю ыхььх 2 Ёхуєы Ёэю уюьюЄюяэр эхъюЄюЁющ ъЁштющ, яхЁхёхўхэшх ъюЄюЁющ ё єърчрээ√ь ышёЄюь ╠хсшєёр  ты хЄё  яЁюёЄющ фєующ, р яхЁхёхўхэшх ё фшёъюь $U'$ шьххЄ тшф яЁюёЄющ фєуш ё эхъюЄюЁ√ь ъюышўхёЄтюь ьрыхэ№ъшї чртшЄъют. ╧Ёш ¤Єюь ёююЄтхЄёЄтє■∙р  Ёхуєы Ёэр  уюьюЄюяш  эхяюфтшцэр т ьрыющ юъЁхёЄэюёЄш $V\cap e_i$ тхЁ°шэ√ $v$.

┬чршьэю єэшўЄюцшь ярЁ√ ёюёхфэшї чртшЄъют ъръ т ёыєўрх 1 ё яюью∙№■ Ёхуєы Ёэющ уюьюЄюяшш, эхяюфтшцэющ эр $V\cap e_i$.
╥ю хёЄ№, т фрээюь ёыєўрх яюыєўрхь, ўЄю эхёЄ уштрхьє■ яхЄы■ $\gamma_1|_{e_i}$ ьюцэю яЁхюсЁрчютрЄ№ ё яюью∙№■ Ёхуєы Ёэющ уюьюЄюяшш, эхяюфтшцэющ эр $V\cap e_i$, т яхЄы■, яхЁхёхўхэшх ъюЄюЁющ ё ышёЄюь ╠хсшєёр $\mathbb R P^2\setminus U'$  ты хЄё  яЁюёЄющ фєующ, р яхЁхёхўхэшх ё фшёъюь $U'$ шьххЄ тшф яЁюёЄющ фєуш ё 0 шыш 1 яюыюцшЄхы№э√ь чртшЄъюь. ╩ръ ш т ёыєўрх 1 яюыєўрхь, ўЄю ъюышўхёЄтю чртшЄъют Ёртэю $I_i$.
╧ю¤Єюьє ъ ¤Єюьє цх тшфє ьюцэю яЁхюсЁрчютрЄ№ ш яхЄы■ $\gamma_2|_{e_i}$, ўЄю ш чртхЁ°рхЄ фюърчрЄхы№ёЄтю ыхьь√.
%\qed
\end{proof}

└тЄюЁ сыруюфрЁхэ ┼.\,└.~╩єфЁ тЎхтющ чр яюёЄрэютъє чрфрўш ш яюыхчэ√х юсёєцфхэш .
%\newpage

%---- ╧╨╚╠┼╨█ ╬╘╬╨╠╦┼═╚▀ ╨╚╤╙═╩╬┬---------------------------------------
%\def\kscale{0.5}
%\makepicone{9a.jpg}{╨шё.\ 1. ┬чршьэюх єэшўЄюцхэшх ярЁ√ яхЄхы№ яЁюЄштюяюыюцэ√ї чэръют}
%
%\makepicone{protask1.jpg}{╨шё.\ 2. ``╧ЁюЄрёъштрэшх ўхЁхч схёъюэхўэюёЄ№''}

%\begin{figure}[htbp]
%\includegraphics[width=0.20\textwidth]{2a.jpg}
%\includegraphics[width=0.30\textwidth]{3a.jpg} 
%\includegraphics[width=0.36\textwidth]{4a.jpg}
%\caption {─тшцхэш  ╨хщфхьхщёЄхЁр Єшяют 1, 2 ш 3} \label {fig:Reid}
%\end {figure}
%
%\begin{figure}[htbp]
%\begin{center}
%\includegraphics[width=0.6\textwidth]{9a.jpg}
%\end{center}
%\caption {┬чршьэюх єэшўЄюцхэшх ярЁ√ ьрыхэ№ъшї яхЄхы№} \label {fig:pair}
%\end {figure}
%
%\begin{figure}[htbp]
%\begin{center}
%\includegraphics[width=0.45\textwidth]{protask1.jpg}
%\hspace*{2mm}
%\includegraphics[width=0.45\textwidth]{protask2.jpg}
%\end{center}
%\caption {``╧ЁюЄрёъштрэшх ўхЁхч схёъюэхўэюёЄ№''} \label {fig:protask}
%\end {figure}
%
%\newpage

\bigskip
\bigskip
\bigskip
Maxim A.~Ivashkovskii

Moscow State University (Moscow, Russia)

frank1581@yandex.ru

\end{document}